\documentclass{amsart}

\newtheorem{theorem}{Theorem}[section]
\newtheorem{lemma}[theorem]{Lemma}

\theoremstyle{definition}

\theoremstyle{remark}

\numberwithin{equation}{section}

\begin{document}
\title{HYPERGEOMETRIC FUNCTIONS OVER $\mathbb{F}_q$ AND TRACES OF FROBENIUS FOR ELLIPTIC CURVES}

\author{Rupam Barman}
\address{Department of Mathematical Sciences, Tezpur University, Napaam-784028, Sonitpur, Assam, India}
\email{rupamb@tezu.ernet.in}
\thanks{The first author thanks Mathematical Institute, University of Heidelberg and Mathematics Center Heidelberg (MATCH), 
where the majority of this research was conducted. He is grateful to John H. Coates, R. Sujatha, Otmar Venjakob, and Anupam Saikia 
for their encouragements. The second author is partially supported by INSPIRE Fellowship of Department 
of Science and Technology, Goverment of India. Finally, the authors thank Ken Ono and the referee for helpful comments.}

\author{Gautam Kalita}
\address{Department of Mathematical Sciences, Tezpur University, Napaam-784028, Sonitpur, Assam, India}
\email{gautamk@tezu.ernet.in}

\subjclass[2000]{Primary 11T24, 11G20}

\date{August, 2011.}


\keywords{Gaussian hypergeometric series, elliptic curves, Frobenius endomorphisms}

\begin{abstract}
We present here explicit relations between the traces of Frobenius endomorphisms of certain families of elliptic curves and 
special values of ${_{2}}F_1$-hypergeometric functions over $\mathbb{F}_q$ for 
$q \equiv 1 ( \text{mod}~6)$ and $q \equiv 1 ( \text{mod}~4)$.
\end{abstract}

\maketitle

\section{Introduction and statement of results}
In this paper, we consider the problem of expressing traces of Frobenius endomorphisms of certain families of elliptic curves 
in terms of hypergeometric functions over finite fields. In \cite{Greene}, Greene introduced the notion of hypergeometric functions 
over finite fields or \emph{Gaussian hypergeometric series} which are analogous to the classical hypergeometric series. Since then, 
many interesting relations between special values of these functions and the number of $\mathbb{F}_p$-points on certain varieties 
have been obtained. For example, Koike \cite{koike} and Ono \cite{ono} gave formulas for the number of $\mathbb{F}_p$-points on 
elliptic curves in terms of special values of Gaussian hypergeometric series. Also, in \cite{BK, BK2} the authors studied this problem 
for certain families of algebraic curves.
\par Recently in \cite{Fuselier}, Fuselier gave formulas for the trace of Frobenius of certain families of elliptic curves which 
involved Gaussian hypergeometric series with characters of order 12 as parameters, under the assumption that $p\equiv 1 (\text{mod}~12)$. 
In \cite{Lennon}, Lennon provided a general formula expressing the number of $\mathbb{F}_q$-points of an elliptic 
curve $E$ with $j(E)\neq 0, 1728$ in terms of values of Gaussian hypergeometric series for $q=p^e \equiv 1 (\text{mod}~ 12)$. 
In \cite{Lennon2}, for $q \equiv 1 (\text{mod}~ 3)$, Lennon also gave formulas for certain elliptic curves involving Gaussian 
hypergeometric series with characters of order 3 as parameters.
\par We begin with some preliminary definitions needed to state our results. Let $q=p^e$ be a power of an odd prime and $\mathbb{F}_q$ 
the finite field of $q$ elements. Extend each character $\chi \in \widehat{\mathbb{F}_q^{\times}}$ to all of $\mathbb{F}_q$ by 
setting $\chi(0):=0$. If $A$ and $B$ are two characters of $\mathbb{F}_q^{\times}$, then ${A \choose B}$ is defined by
\begin{align}\label{eq0}
{A \choose B}:=\frac{B(-1)}{q}J(A,\overline{B})=\frac{B(-1)}{q}\sum_{x \in \mathbb{F}_q}A(x)\overline{B}(1-x),
\end{align}
where $J(A, B)$ denotes the usual Jacobi sum and $\overline{B}$ is the inverse of $B$.
\par
Recall the definition of the Gaussian hypergeometric series over $\mathbb{F}_q$ first defined by Greene in \cite{Greene}. 
For any positive integer $n$ and characters $A_0, A_1,\ldots, A_n$ and $B_1, B_2,\ldots, B_n \in \widehat{\mathbb{F}_q^{\times}}$, 
the Gaussian hypergeometric series ${_{n+1}}F_n$ is defined to be
\begin{align}\label{eq00}
{_{n+1}}F_n\left(\begin{array}{cccc}
                A_0, & A_1, & \cdots, & A_n\\
                 & B_1, & \cdots, & B_n
              \end{array}\mid x \right):=\frac{q}{q-1}\sum_{\chi}{A_0\chi \choose \chi}{A_1\chi \choose B_1\chi}
\cdots {A_n\chi \choose B_n\chi}\chi(x),
\end{align}
where the sum is over all characters $\chi$ of $\mathbb{F}_q^{\times}$.
\par Throughout the paper, we consider an elliptic curve $E_{a,b}$ over $\mathbb{F}_q$ in Weierstrass form as
\begin{align}\label{eq100}
E_{a,b}: y^2=x^3+ax+b.
\end{align}
If we denote by $a_q(E_{a,b})$ the trace of the Frobenius endomorphism on $E_{a,b}$, then
\begin{align}\label{eq101}
a_q(E_{a,b})=q+1-\#E_{a,b}(\mathbb{F}_q),
\end{align}
where $\#E_{a,b}(\mathbb{F}_q)$ denotes the number of $\mathbb{F}_q$-points on $E_{a,b}$ including the point at infinity. 
In the following theorems, we express $a_q(E_{a,b})$ in terms of Gaussian hypergeometric series.
\begin{theorem}\label{mt1} Let $q=p^e$, $p>0$ a prime and $q\equiv1~(mod~6)$. In addition, let $a$ be non-zero 
and $(-a/3)$ a quadratic residue modulo $q$. If $T \in \widehat{\mathbb{F}_q^{\times}}$ is a generator of the character group, 
then the trace of the Frobenius on $E_{a,b}$ can be expressed as
\begin{align}
a_q(E_{a,b})=-qT^{\frac{q-1}{2}}(-k)~{_{2}}F_1\left(\begin{array}{cccc}
                T^{\frac{q-1}{6}}, & T^{\frac{5(q-1)}{6}}\\
                 & \epsilon
              \end{array}\mid -\frac{k^3+ak+b}{4k^3} \right),\nonumber
\end{align}
where $\epsilon$ is the trivial character of $\mathbb{F}_q$ and $k\in \mathbb{F}_q$ satisfies $3k^2+a=0$.
\end{theorem}
\begin{theorem}\label{mt2} Let $q=p^e$, $p>0$ a prime, $q\neq 9$ and $q\equiv1~(mod~4)$. Also assume that $x^3+ax+b=0$ 
has a non-zero solution in $\mathbb{F}_q$ and $T\in \widehat{\mathbb{F}_q^{\times}}$ is a generator of the character group. 
The trace of the Frobenius on $E_{a,b}$ can be expressed as
\begin{align}
a_q(E_{a,b})=-qT^{\frac{q-1}{2}}(6h)T^{\frac{q-1}{4}}(-1)~{_{2}}F_1\left(\begin{array}{cccc}
                T^{\frac{q-1}{4}}, & T^{\frac{3(q-1)}{4}}\\
                 & \epsilon
              \end{array}\mid \frac{12h^2+4a}{9h^2} \right),\nonumber
\end{align}
where $\epsilon$ is the trivial character of $\mathbb{F}_q$ and $h\in\mathbb{F}^{\times}_q$ satisfies $h^3+ah+b=0$.
\end{theorem}

\section{Preliminaries}
Define the additive character $\theta: \mathbb{F}_q \rightarrow \mathbb{C}^{\times}$ by
\begin{align}
\theta(\alpha)=\zeta^{\text{tr}(\alpha)}
\end{align}
where $\zeta=e^{2\pi i/p}$ and $\text{tr}: \mathbb{F}_q \rightarrow \mathbb{F}_q$ is the trace map given by
$$\text{tr}(\alpha)=\alpha + \alpha^p + \alpha^{p^2}+ \cdots + \alpha^{p^{e-1}}.$$
For $A\in \widehat{\mathbb{F}_q^\times}$, the \emph{Gauss sum} is defined by
\begin{align}
G(A):=\sum_{x\in \mathbb{F}_q}A(x)\zeta^{\text{tr}(x)}=\sum_{x\in \mathbb{F}_q}A(x)\theta(x).
\end{align}
We let $T$ denote a fixed generator of $\widehat{\mathbb{F}_q^\times}$. We also denote by $G_m$ the Gauss sum $G(T^m)$. \par 
The \emph{orthogonality relations} for multiplicative characters are listed in the following lemma.
\begin{lemma}\emph{(\cite{ireland} Chapter 8).}\label{lemma2}
Let $\epsilon$ be the trivial character. Then
\begin{enumerate}
\item $\sum_{x\in\mathbb{F}_q}T^n(x)=\left\{
                                  \begin{array}{ll}
                                    q-1 & \hbox{if~ $T^n=\epsilon$;} \\
                                    0 & \hbox{if ~~$T^n\neq\epsilon$.}
                                  \end{array}
                                \right.$
\item $\sum_{n=0}^{q-2}T^n(x)~~=\left\{
                            \begin{array}{ll}
                              q-1 & \hbox{if~~ $x=1$;} \\
                              0 & \hbox{if ~~$x\neq1$.}
                            \end{array}
                          \right.$
\end{enumerate}
\end{lemma}
Using orthogonality, we have the following lemma.
\begin{lemma}\emph{(\cite{Fuselier} Lemma 2.2).}\label{lemma1}
For all $\alpha \in \mathbb{F}_q^{\times}$, $$\theta(\alpha)=\frac{1}{q-1}\sum_{m=0}^{q-2}G_{-m}T^m(\alpha).$$
\end{lemma}
The following two lemmas on Gauss sum will be useful in the proof of our results.
\begin{lemma}\emph{(\cite{Greene} Eqn. 1.12).}\label{new}
If $i\in \mathbb{Z}$ and $T^i\neq \epsilon$, then $$G_iG_{-i}=qT^i(-1).$$
\end{lemma}
\begin{lemma}\emph{(Davenport-Hasse Relation \cite{Lang}).}\label{lemma3}
Let $m$ be a positive integer and let $q=p^e$ be a prime power such that $q\equiv 1 (\text{mod}~m)$. 
For multiplicative characters $\chi, \psi \in \widehat{\mathbb{F}_q^\times}$, we have
\begin{align}
\prod_{\chi^m=1}G(\chi \psi)=-G(\psi^m)\psi(m^{-m})\prod_{\chi^m=1}G(\chi).
\end{align}
\end{lemma}

\section{Proof of the results}
Theorem \ref{mt1} will follow as a consequence of the next theorem. We consider an elliptic curve $E_1$ over $\mathbb{F}_q$ 
in the form
\begin{align}\label{eq102}
E_1: y^2=x^3+cx^2+d,
\end{align}
where $c \neq 0$. The trace of the Frobenius endomorphism on $E_1$ is given by
\begin{align}\label{eq103}
a_q(E_1)=q+1-\#E_1(\mathbb{F}_q).
\end{align}
We express the trace of Frobenius on the curve $E_1$ as a special value of a hypergeometric function in the following way.
\begin{theorem}\label{theorem1}
Let $q=p^e$, $p>0$ a prime and $q\equiv1~(mod~6)$. If $T \in \widehat{\mathbb{F}_q^{\times}}$ is a generator of the character group, 
then the trace of the Frobenius on $E_1$ is given by
\begin{align}
a_q(E_1)=-qT^{\frac{q-1}{2}}(-3c)~{_{2}}F_1\left(\begin{array}{cccc}
                T^{\frac{q-1}{6}}, & T^{\frac{5(q-1)}{6}}\\
                 & \epsilon
              \end{array}\mid -\frac{27d}{4c^3} \right),\nonumber
\end{align}
where $\epsilon$ is the trivial character of $\mathbb{F}_q$.
\end{theorem}
\begin{proof} The method of this proof follows similarly to that given in \cite{Fuselier}.
Let $$P(x,y)=x^3+cx^2+d-y^2$$ and denote by $\#E_1(\mathbb{F}_q)$ the number of points on the curve $E_1$ over $\mathbb{F}_q$ 
including the point at infinity. Then $$\#E_1(\mathbb{F}_q)-1=\#\{(x,y)\in \mathbb{F}_q\times \mathbb{F}_q : P(x,y)=0\}.$$
Using the elementary identity from \cite{ireland} \begin{align}\label{eq4}
\sum_{z\in \mathbb{F}_q}\theta(zP(x,y))=\left\{
                                            \begin{array}{ll}
                                              q & \hbox{if $P(x,y)=0$;} \\
                                              0 & \hbox{if $P(x,y)\neq 0,$}
                                            \end{array}
                                          \right.
\end{align}
we obtain
\begin{align}\label{eq1}
q\cdot(\#E_1(\mathbb{F}_q)-1)&=\sum_{x,y,z\in \mathbb{F}_q}\theta(zP(x,y))\nonumber\\
&=q^2+\sum_{z\in\mathbb{F}_q^\times}\theta(zd)+\sum_{y,z\in\mathbb{F}_q^\times}\theta(zd)\theta(-zy^2)+
\sum_{x,z\in\mathbb{F}_q^\times}\theta(zd)\theta(zx^3)\theta(zcx^2)\nonumber\\
&\hspace{.5cm}+\sum_{x,y,z\in \mathbb{F}_q^\times}\theta(zd)\theta(zx^3)\theta(zcx^2)\theta(-zy^2)\nonumber\\
&:=q^2+A+B+C+D.
\end{align}
Now using Lemma \ref{lemma1} and then applying Lemma \ref{lemma2} repeatedly for each term of \eqref{eq1}, 
we deduce that $$A=\frac{1}{q-1}\sum_{z\in \mathbb{F}_q^\times}\sum_{l=0}^{q-2}G_{-l}T^l(zd)=\frac{1}{q-1}\sum_{l=0}^{q-2}G_{-l}T^l(d)
\sum_{z\in \mathbb{F}_q^\times}T^l(z)=G_0=-1.$$ 
Similarly,
\begin{align}
B&=\frac{1}{(q-1)^2}\sum_{l,m=0}^{q-2}G_{-l}G_{-m}T^l(d)T^m(-1)\sum_{y\in \mathbb{F}_q^\times}T^{2m}(y)
\sum_{z\in \mathbb{F}_q^\times}T^{l+m}(z)\nonumber\\
&=1+G_{\frac{q-1}{2}}G_{-\frac{q-1}{2}}T^{\frac{q-1}{2}}(d)T^{\frac{q-1}{2}}(-1).\notag
\end{align}
Using Lemma \ref{new} for $i=\frac{q-1}{2}$, we deduce that
\begin{align}
B&=1+qT^{\frac{q-1}{2}}(-1)T^{\frac{q-1}{2}}(d)T^{\frac{q-1}{2}}(-1)\notag\\
&=1+qT^{\frac{q-1}{2}}(d).\notag
\end{align}
Expanding the next term, we have
\begin{align}
C&=\frac{1}{(q-1)^3}\sum_{l,m,n=0}^{q-2}G_{-l}G_{-m}G_{-n}T^l(d)T^n(c)\sum_{z\in \mathbb{F}_q^\times}T^{l+m+n}(z)
\sum_{x\in \mathbb{F}_q^\times}T^{3m+2n}(x).\nonumber
\end{align}
Finally,
\begin{align}
D&=\frac{1}{(q-1)^4}\sum_{l,m,n,k=0}^{q-2}G_{-l}G_{-m}G_{-n}G_{-k}T^l(d)T^n(c)T^k(-1)\times \nonumber\\
&\hspace{.5cm}\sum_{z\in \mathbb{F}_q^\times}T^{l+m+n+k}(z)\sum_{x\in \mathbb{F}_q^\times}T^{3m+2n}(x)
\sum_{z\in\mathbb{F}_q^\times}T^{2k}(z)\nonumber.
\end{align}
The innermost sum of $D$ is nonzero only when $k=0$ or $k=\frac{q-1}{2}$. Using the fact that $G_0=-1$, we obtain
\begin{align}
D=-C+D_{\frac{q-1}{2}},\nonumber
\end{align}
where
\begin{align}
D_{\frac{q-1}{2}}&=\frac{1}{(q-1)^3}\sum_{l,m,n=0}^{q-2}G_{-l}G_{-m}G_{-n}
G_{\frac{q-1}{2}}T^l(d)T^n(c)T^{\frac{q-1}{2}}(-1)\times \nonumber\\
&\hspace{.5cm}\sum_{z\in \mathbb{F}_q^\times}T^{l+m+n+\frac{q-1}{2}}(z)\sum_{x\in \mathbb{F}_q^\times}T^{3m+2n}(x),\nonumber
\end{align}
which is zero unless $m=-\frac{2}{3}n$ and $n=-3l-\frac{3(q-1)}{2}$. Since $G_{3l+\frac{3(q-1)}{2}}=G_{3l+\frac{q-1}{2}}$ and
 $G_{-2l-(q-1)}=G_{-2l}$, we have
\begin{align}
D_{\frac{q-1}{2}}=\frac{1}{q-1}\sum_{l=0}^{q-2}G_{-l}G_{-2l}G_{3l+\frac{q-1}{2}}
G_{\frac{(q-1)}{2}}T^l(d)T^{-3l+\frac{q-1}{2}}(c)T^{\frac{q-1}{2}}(-1).\nonumber
\end{align}
Using Davenport-Hasse relation \eqref{lemma3} for $m=2, \psi=T^{-l}$ and $m=3, \psi=T^{l+\frac{q-1}{6}}$ respectively, 
we deduce that $$G_{-2l}=\frac{G_{-l}G_{-l-\frac{q-1}{2}}}{G_{\frac{q-1}{2}}T^l(4)} ~~~~~~~ \text{and} ~~~~~~~ G_{3l+\frac{q-1}{2}}
=\frac{G_{l+\frac{q-1}{6}}G_{l+\frac{q-1}{2}}
G_{l+\frac{5(q-1)}{6}}}{qT^{-l-\frac{q-1}{6}}(27)}.$$
Therefore,
\begin{align}
D_{\frac{q-1}{2}}&=\frac{T^{\frac{q-1}{2}}(-3c)}{q(q-1)}\sum_{l=0}^{q-2}G_{-l}G_{-l}G_{-l-\frac{q-1}{2}}
G_{l+\frac{q-1}{6}}G_{l+\frac{q-1}{2}}G_{l+\frac{5(q-1)}{6}}T^l\left(\frac{27d}{4c^3}\right).\nonumber
\end{align}
Now, if $T^{m-n}\neq\epsilon$, then we have
\begin{align}\label{eq3}
G_mG_{-n}=q{T^m \choose T^n}G_{m-n}T^n(-1).
\end{align}
Replacing $l$ by $l-\frac{q-1}{2}$ and using \eqref{eq3}, we obtain
\begin{align}
D_{\frac{q-1}{2}}&=\frac{qT^{\frac{q-1}{2}}(-3c)}{q-1}
\sum_{l=0}^{q-2}G_{l}G_{-l}{T^{l-\frac{q-1}{3}} \choose T^{l-\frac{q-1}{2}}}
G_{\frac{q-1}{6}}{T^{l+\frac{q-1}{3}} \choose T^{l-\frac{q-1}{2}}}
G_{\frac{5(q-1)}{6}}T^{l-\frac{q-1}{2}}\left(\frac{27d}{4c^3}\right).\nonumber
\end{align}
Plugging the facts that if $l\neq0$ then $G_lG_{-l}=qT^l(-1)$ and if $l=0$ then $G_lG_{-l}=qT^l(-1)-(q-1)$ 
in appropriate identities for each $l$, we deduce that
\begin{align}
D_{\frac{q-1}{2}}&=\frac{q^3T^{\frac{q-1}{6}}(-1)T^{\frac{q-1}{2}}(-3c)}{q-1}
\sum_{l=0}^{q-2}{T^{l-\frac{q-1}{3}} \choose T^{l-\frac{q-1}{2}}}{T^{l+\frac{q-1}{3}} \choose T^{l-\frac{q-1}{2}}}T^{l-\frac{q-1}{2}}
\left(\frac{27d}{4c^3}\right)T^l(-1)\nonumber\\ 
&\hspace{.5cm}-q^2T^{\frac{q-1}{6}}(-1)T^{\frac{q-1}{2}}(-3c){T^{\frac{2(q-1)}{3}} \choose T^{\frac{q-1}{2}}}
{T^{\frac{q-1}{3}} \choose T^{\frac{q-1}{2}}}T^{\frac{q-1}{2}}\left(\frac{27d}{4c^3}\right).\nonumber
\end{align}
Replacing $l$ by $l+\frac{q-1}{2}$ in the first term and simplifying the second term, we obtain
\begin{align}
D_{\frac{q-1}{2}}&=\frac{q^3T^{\frac{q-1}{2}}(-3c)}{q-1}\sum_{l=0}^{q-2}
{T^{l+\frac{q-1}{6}} \choose T^l}{T^{l+\frac{5(q-1)}{6}} \choose T^l}T^l\left(-\frac{27d}{4c^3}\right)\nonumber\\ 
&\hspace{.5cm}-q^2T^{\frac{q-1}{2}}(d)\frac{{G_\frac{2(q-1)}{3}}G_{\frac{q-1}{2}}G_\frac{q-1}{3}
G_{\frac{q-1}{2}}}{q^2G_{\frac{q-1}{6}}G_{\frac{5(q-1)}{6}}}\nonumber\\
&=q^2T^{\frac{q-1}{2}}(-3c){_{2}}F_1\left(\begin{array}{cccc}
                T^{\frac{q-1}{6}}, & T^{\frac{5(q-1)}{6}}\\
                 & \epsilon
              \end{array}\mid -\frac{27d}{4c^3} \right) -qT^{\frac{q-1}{2}}(d).\nonumber
\end{align}
Putting the values of $A, B, C, D$ all together in \eqref{eq1} gives
\begin{align}
q\cdot (\#E_1(\mathbb{F}_q)-1)&=q^2+q^2T^{\frac{q-1}{2}}(-3c){_{2}}F_1\left(\begin{array}{cccc}
                T^{\frac{q-1}{6}}, & T^{\frac{5(q-1)}{6}}\\
                 & \epsilon
              \end{array}\mid -\frac{27d}{4c^3}\right).\nonumber
\end{align}
Since $a_q(E_1)=q+1-\#E_1(\mathbb{F}_q)$, we have completed the proof of the Theorem.
\end{proof}
\noindent
\textbf{Proof of Theorem \ref{mt1}.}
Since $a\neq 0$ and $(-a/3)$ is quadratic residue modulo $q$, we find $k \in \mathbb{F}_q^{\times}$ such that $3k^2+a=0$.
 A change of variables $(x, y) \mapsto (x+k, y)$ takes the elliptic curve $E_{a, b}: y^2=x^3+ax+b$ to
\begin{align}\label{curve1}
E'_{a, b}: y^2=x^3+3kx^2+(k^3+ak+b).
\end{align}
Clearly $a_{q}(E_{a,b})=a_q(E'_{a,b}).$
Since $3k\neq 0$, using Theorem \ref{theorem1} for the elliptic curve $E'_{a,b}$, we complete the proof.\hfill $\Box$

\par We now prove a result for $q \equiv 1 (\text{mod}~4)$ similar to Theorem \ref{theorem1} and Theorem \ref{mt2} 
will follow from this result.
\begin{theorem}\label{theorem2}
Let $q=p^e$, $p>0$ a prime and $q\equiv1~(mod~4)$. Let $E_2$ be an elliptic curve over $\mathbb{F}_q$ 
defined as $$E_2 : y^2=x^3+fx^2+gx$$ such that $f\neq 0$. If $T \in \widehat{\mathbb{F}_q^{\times}}$ is a generator of 
the character group, then the trace of the Frobenius on $E_2$ is given by
\begin{align}
a_q(E_2)=-qT^{\frac{q-1}{2}}(2f)T^{\frac{q-1}{4}}(-1){_{2}}F_1\left(\begin{array}{cccc}
                T^{\frac{q-1}{4}}, & T^{\frac{3(q-1)}{4}}\\
                 & \epsilon
              \end{array}\mid \frac{4g}{f^2} \right),\nonumber
\end{align}
where $\epsilon$ is the trivial character of $\mathbb{F}_q$.
\end{theorem}
\begin{proof}
We have $$\#E_2(\mathbb{F}_q)-1=\#\{(x,y)\in \mathbb{F}_q\times \mathbb{F}_q : P(x,y)=0\},$$
where $$P(x,y)=x^3+fx^2+gx-y^2.$$
Using \eqref{eq4}, we express the number of points as
\begin{align}\label{eq2}
q\cdot(\#E_2(\mathbb{F}_q)-1)&=\sum_{x,y,z\in \mathbb{F}_q}\theta(zP(x,y))\nonumber\\
&=q^2 +\sum_{z\in \mathbb{F}_q^\times}\theta(0) +\sum_{y,z\in \mathbb{F}_q^\times}\theta(-zy^2) +
\sum_{x,z\in \mathbb{F}_q^\times}\theta(zx^3)\theta(zfx^2)\theta(zgx)\nonumber\\ &\hspace{.5cm}+
\sum_{x,y,z\in \mathbb{F}_q^\times}\theta(zx^3)\theta(zfx^2)\theta(zgx)\theta(-zy^2)\nonumber\\
&:=q^2+(q-1)+A+B+C.
\end{align}
Now, following the same procedure as followed in the proof of Theorem \eqref{theorem1}, we deduce that
\begin{align}
A&=-(q-1)\nonumber\\
B&=\frac{1}{(q-1)^3}\sum_{l,m,n=0}^{q-2}G_{-l}G_{-m}G_{-n}T^m(f)T^n(g)\sum_{z\in \mathbb{F}_q^\times}T^{l+m+n}(z)\sum_{x\in
\mathbb{F}_q^\times}T^{3l+2m+n}(x)\nonumber\\
C&=-\frac{1}{(q-1)^3}\sum_{l,m,n=0}^{q-2}G_{-l}G_{-m}G_{-n}T^m(f)T^n(g)\sum_{z\in \mathbb{F}_q^\times}T^{l+m+n}(z)\sum_{x\in
\mathbb{F}_q^\times}T^{3l+2m+n}(x)\nonumber\\
&\hspace{.5cm}+\frac{1}{(q-1)^3}\sum_{l,m,n=0}^{q-2}G_{-l}G_{-m}G_{-n}G_{\frac{q-1}{2}}T^m(f)T^n(g)
\sum_{z\in \mathbb{F}_q^\times}T^{l+m+n+\frac{q-1}{2}}(z)\sum_{x\in \mathbb{F}_q^\times}T^{3l+2m+n}(x).\nonumber
\end{align}
Substituting the values of $A$, $B$, $C$ all together in \eqref{eq2} and simplifying after using Lemma \ref{lemma2}, we obtain
\begin{align}\label{eq100}
q\cdot(\#E_2(\mathbb{F}_q)-1)&=q^2+\frac{G_{\frac{q-1}{2}}T^{\frac{q-1}{2}}(f)}{q-1}\sum_{l=0}^{q-2}
G_{-l}G_{2l+\frac{q-1}{2}}G_{-l}T^l\left(\frac{g}{f^2}\right).
\end{align}
The Davenport-Hasse relation \eqref{lemma3} with $m=2, \psi=T^{l+\frac{q-1}{4}}$ yields
\begin{align}\label{new1}
G_{2l+\frac{q-1}{2}}=\frac{G_{l+\frac{q-1}{4}}G_{l+\frac{3(q-1)}{4}}}{G_{\frac{q-1}{2}}}T^{l-\frac{q-1}{4}}(4).
\end{align}
Using \eqref{new1} and then \eqref{eq3} in \eqref{eq100}, we have
\begin{align}
q\cdot(\#E_2(\mathbb{F}_q)-1)&=q^2+\frac{q^3T^{\frac{q-1}{2}}(2f)
T^{\frac{q-1}{4}}(-1)}{q-1}\sum_{l=0}^{q-2}{T^{l+\frac{q-1}{4}} \choose T^l}{T^{l+\frac{3(q-1)}{4}} 
\choose T^l}T^l\left(\frac{4g}{f^2}\right)\nonumber\\
&=q^2+q^2T^{\frac{q-1}{2}}(2f)T^{\frac{q-1}{4}}(-1){_{2}}F_1\left(\begin{array}{cccc}
                T^{\frac{q-1}{4}}, & T^{\frac{3(q-1)}{4}}\\
                 & \epsilon
              \end{array}\mid \frac{4g}{f^2} \right)\nonumber
\end{align}
and then using the relation $a_q(E_2)=q+1-\#E_2(\mathbb{F}_q)$, we complete the proof.
\end{proof}
\noindent
\textbf{Proof of Theorem \ref{mt2}.}
Since $x^3+ax+b=0$ has a non-zero solution in $\mathbb{F}_q$, let $h\in\mathbb{F}_q^{\times}$ be such that $h^3+ah+b=0$. 
A change of variables $(x, y) \mapsto (x+h, y)$ takes the elliptic curve $E_{a, b}: y^2=x^3+ax+b$ to
\begin{align}\label{curve2}
E''_{a, b}: y^2=x^3+3hx^2+(3h^2+a)x.
\end{align}
Since $a_{q}(E_{a,b})=a_q(E''_{a,b})$ and $3h\neq 0$, using Theorem \ref{theorem2} for the elliptic curve $E''_{a,b}$, 
we complete the proof.\hfill $\Box$

\bibliographystyle{amsplain}

\begin{thebibliography}{10}
\bibitem{BK} R. Barman and G. Kalita, \textit{Hypergeometric functions and a family of algebraic curves}, Ramanujan J.
(to appear).

\bibitem{BK2}
R. Barman and G. Kalita, \textit{Certain values of Gaussian hypergeometric series and a family of algebraic curves}, 
Int. J. Number Theory (to appear).

\bibitem {Fuselier} J. Fuselier, \textit{Hypergeometric functions over $\mathbb{F}_p$ and relations to elliptic curves and modular forms}, 
Proc. Amer. Math. Soc. \textbf{138} (2010), 109--123.

\bibitem {Greene} J. Greene, \textit{Hypergeometric functions over finite fields}, Trans. Amer. Math. Soc.
\textbf{301} (1987), no. 1, 77--101.

\bibitem{ireland} K. Ireland and M. Rosen, \textit{A Classical Introduction to Modarn Number Theory}, 2nd ed., Graduate Texts in Mathematics, vol. 84, 
Springer-Verlag, New York, 1990.

\bibitem{koike} M. Koike, \textit{Hypergeometric series over finite fields and Ap\'{e}ry numbers}, Hiroshima Math. J. \textbf{22} (1992), 461-467.

\bibitem {Lang} S. Lang, \textit{Cyclotomic Fields I and II}, Graduate Texts in Mathematics, vol. 121, Springer-Verlag, New York, 1990.

\bibitem {Lennon} C. Lennon, \textit{Gaussian hypergeometric evaluations of traces of Frobenius for elliptic curves}, 
Proc. Amer. Math. Soc. \textbf{139} (2011), 1931--1938.

\bibitem {Lennon2} C. Lennon, \textit{Trace formulas for Hecke operators, Gaussian hypergeometric functions, and the modularity of a threefold}, 
J. Number Theory, \textbf{131} (2011), no. 12, 2320--2351.

\bibitem{ono} K. Ono, \textit{Values of Gaussian hypergeometric series}, Trans. Amer. Math. Soc. \textbf{350} (1998), no. 3, 1205--1223.

\end{thebibliography}

\end{document}